\documentclass[12pt]{amsart}
\usepackage{amscd,amsmath,amsthm,amssymb,graphics}
\usepackage{pstcol,pst-plot,pst-3d}
\usepackage{lmodern,pst-node}
\usepackage[dvips]{graphicx}
\usepackage{multicol}
\usepackage{epic,eepic}
\usepackage{amsfonts,amssymb,amscd,amsmath,enumitem,verbatim}
\psset{unit=0.7cm,linewidth=0.8pt,arrowsize=2.5pt 4}

\newpsstyle{fatline}{linewidth=1.5pt}
\newpsstyle{fyp}{fillstyle=solid,fillcolor=verylight}
\definecolor{verylight}{gray}{0.97}
\definecolor{light}{gray}{0.9}
\definecolor{medium}{gray}{0.85}
\definecolor{dark}{gray}{0.6}

\usepackage{lineno}

\def\frk{\frak}               

\def\Phi{{\frk n}}
\def\Phi{{\frk N}}
\def\MF{{\mathcal F}}

\def\MS{{\mathcal S}}

\def\MG{{\mathcal G}}

\def\opn#1#2{\def#1{\operatorname{#2}}} 
\opn\chara{char} \opn\length{\ell} \opn\pd{pd} \opn\rk{rk}
\opn\projdim{proj\,dim} \opn\injdim{inj\,dim} \opn\rank{rank}
\opn\depth{depth} \opn\grade{grade} \opn\height{height}
\opn\embdim{emb\,dim} \opn\codim{codim}

\opn\Tr{Tr} \opn\bigrank{big\,rank}
\opn\superheight{superheight}\opn\lcm{lcm}
\opn\trdeg{tr\,deg}
\opn\reg{reg} \opn\lreg{lreg} \opn\ini{in} \opn\lpd{lpd}
\opn\size{size}\opn\bigsize{bigsize}
\opn\cosize{cosize}\opn\bigcosize{bigcosize}
\opn\sdepth{sdepth}\opn\sreg{sreg}
\opn\link{link}\opn\fdepth{fdepth}
\opn\div{div} \opn\Div{Div} \opn\cl{cl} \opn\Cl{Cl}
\let\epsilon\varepsilon
\let\phi=\varphi
\let\kappa=\varkappa
\opn\Spec{Spec} \opn\Supp{Supp} \opn\supp{supp} \opn\Sing{Sing}
\opn\Ass{Ass} \opn\Min{Min}\opn\Mon{Mon} \opn\dstab{dstab} \opn\astab{astab}
\opn\Syz{Syz}
\opn\Ann{Ann} \opn\Rad{Rad} \opn\Soc{Soc}
\opn\Im{Im} \opn\Ker{Ker} \opn\Coker{Coker} \opn\Am{Am}
\opn\Hom{Hom} \opn\Tor{Tor} \opn\Ext{Ext} \opn\End{End}
\opn\Aut{Aut} \opn\id{id}

\opn\nat{nat}
\opn\pff{pf}
\opn\Pf{Pf} \opn\GL{GL} \opn\SL{SL} \opn\mod{mod} \opn\ord{ord}
\opn\Gin{Gin} \opn\Hilb{Hilb}\opn\sort{sort}
\opn\initial{init}
\opn\ende{end}
\opn\height{height}
\opn\type{type}
\opn\set{set}
\opn\aff{aff} \opn\con{conv} \opn\relint{relint} \opn\st{st}
\opn\lk{lk} \opn\cn{cn} \opn\core{core} \opn\vol{vol}
\opn\link{link} \opn\star{star}\opn\lex{lex}
\opn\gr{gr}

\def\pot#1#2{#1[\kern-0.28ex[#2]\kern-0.28ex]}

\opn\dirlim{\underrightarrow{\lim}}
\opn\inivlim{\underleftarrow{\lim}}
\let\to=\rightarrow

\def\Implies{\ifmmode\Longrightarrow \else
        \unskip${}\Longrightarrow{}$\ignorespaces\fi}
\def\implies{\ifmmode\Rightarrow \else
        \unskip${}\Rightarrow{}$\ignorespaces\fi}
\def\iff{\ifmmode\Longleftrightarrow \else
        \unskip${}\Longleftrightarrow{}$\ignorespaces\fi}

\let\:=\colon
 \theoremstyle{plain}
\newtheorem{Theorem}{Theorem}[section]
 \newtheorem{Lemma}[Theorem]{Lemma}
 \newtheorem{Corollary}[Theorem]{Corollary}
 \newtheorem{Proposition}[Theorem]{Proposition}

 \theoremstyle{definition}
 \newtheorem{Definition}[Theorem]{Definition}

\let\epsilon\varepsilon
\let\kappa=\varkappa
\opn\dis{dis}
\def\pnt{{\raise0.5mm\hbox{\large\bf.}}}

\opn\Lex{Lex}

\begin{document}
\title{t-spread strongly stable monomial ideals}
\author {Viviana Ene, J\"urgen Herzog,  Ayesha Asloob Qureshi}

\address{Viviana Ene, Faculty of Mathematics and Computer Science, Ovidius University, Bd.\ Mamaia 124,
 900527 Constanta, Romania}  \email{vivian@univ-ovidius.ro}

\address{J\"urgen Herzog, Fachbereich Mathematik, Universit\"at Duisburg-Essen, Fakult\"at f\"ur Mathematik, 45117 Essen, Germany} \email{juergen.herzog@uni-essen.de}

\address{Ayesha Asloob Qureshi, Sabanc\i \; University, Faculty of Engineering and Natural Sciences, Orta Mahalle, Tuzla 34956, Istanbul, Turkey}\email{aqureshi@sabanciuniv.edu}

\thanks{Part of this paper was written while the second author visited Faculty of Engineering and Natural Sciences at Sabanci University.}

\begin{abstract}
We introduce the concept of $t$-spread monomials and $t$-spread strongly stable ideals. These concepts are a natural generalization of strongly stable and squarefree strongly stable ideals. For the study of this class of ideals we use the $t$-fold stretching operator. It is shown that $t$-spread strongly stable ideals  are componentwise linear. Their height, their graded Betti numbers and their generic initial ideal are determined. We also consider the toric rings whose generators come from $t$-spread principal Borel ideals.
\end{abstract}

\thanks{}

\maketitle
\section*{Introduction}
Among the monomial ideals, the squarefree monomial ideals play a distinguished role as they are linked in many ways to combinatorial objects such as simplicial complexes and graphs. Squarefree monomials in a polynomial ring $K[x_1, \ldots, x_n]$ which generate these ideals are monomials of the form $x_{i_1}\cdots x_{i_d}$ with $i_1 < i_2< \cdots < i_d$. In this paper, we call a monomial $x_{i_1}x_{i_2}\cdots x_{i_d}$ with $i_1\leq i_2\leq\cdots \leq i_d$  {\em $t$-spread,} if $i_j- i_{j-1}\geq t$ for $2\leq j \leq n$. Note that, any monomial is $0$-spread, while the squarefree monomials are  $1$-spread.

A  monomial ideal in $S$ is called  a {\em $t$-spread monomial ideal},  if it is generated by  $t$-spread monomials. For example,  $I=(x_1x_4x_8,x_2x_5x_8,x_1x_5x_9,x_2x_6x_9,x_4x_9) \subset K[x_1, \ldots, x_9]$ is a $3$-spread monomial ideal, but not $4$-spread, because $x_2x_5x_8$ is not a $4$-spread monomial. Note that $2$--spread monomial ideals appear as initial ideals for the defining ideals of the fiber cones of monomial ideals in two variables \cite{HQS}.

There is  a well-known deformation, called polarization, which assigns to each monomial ideal a squarefree monomial ideal, preserving all homological properties of these ideals. In this way, many problems regarding monomial ideals can be reduced to the study of squarefree monomial ideals. On the other hand, in shifting theory, in particular for symmetric algebraic shifting, one used another operator, called {\em stretching operator}, see \cite{K}, \cite{HHBook}. To transform an arbitrary monomial $u=x_{i_1}\cdots x_{i_d}$ with $i_1 \leq  i_2\leq \cdots \leq i_d$  into a squarefree monomial, one defines the stretched monomial  $\sigma(u)=x_{i_1}x_{i_2+1}x_{i_3+2}\cdots x_{i_d+(d-1)}$. Let $I$ be a monomial ideal and $G(I)=\{u_1, \ldots, u_m\}$ be the unique minimal monomial set of generators of $I$. Then $I^{\sigma}$ is defined to be the ideal with $G(I^{\sigma})=\{\sigma(u_1), \ldots, \sigma(u_m)\}$.

In contrast to polarization, the stretching operator is not a deformation and in general does not preserve any of the homological properties of the ideal. For example, if $I=(x_1^2, x_2^2)$, then $I^{\sigma}= (x_1x_2, x_2x_3)$. In this example $I$ is a complete intersection, but $I^{\sigma}$ does not have this property. In fact, $I^{\sigma}$ has a linear resolution. Applying again the operator $\sigma$ to $I^{\sigma}$, we obtain the ideal $I^{\sigma^2}=(x_1x_3, x_2x_4)$ which again is a complete intersection.

It can be easily seen that the $t$-fold iterated operator  $\sigma^t$ establishes a bijection
between all monomials in the polynomial ring $T=K[x_1, x_2, \ldots]$ and all $t$-spread monomials in $T$; see Corollary~\ref{bijection}. While, in general $I$ and $I^{\sigma}$ may have different graded Betti numbers, it turns out that the graded Betti numbers coincide when $I$ is a strongly stable ideal. This fact has been used in shifting theory to define symmetric algebraic shifting; see for example \cite[Section 11.2.2]{HHBook}. More generally, as one of the main result of this paper, we show that  $I$ is a $t$-spread strongly stable ideal if and only if  $I^{\sigma}$ is a $t+1$-spread strongly stable ideal (Proposition~\ref{equal}), and $I$ and $I^{\sigma}$ have the same graded Betti numbers; see Theorem~\ref{betti}. The concept $t$-spread strongly stable ideal generalizes the concepts of strongly stable and squarefree strongly stable ideals, and is defined as follows: a monomial ideal $I$  is called {\em $t$-spread strongly stable}, if for all $t$-spread monomials $u\in I$, all $j\in \supp(u)$ and all $i<j$ such that $x_i(u/x_{j})$ is $t$-spread, it follows that  $x_i(u/x_j)\in I$.

By using Theorem~\ref{betti} and a well known result of Eliahou-Kervaire \cite{EK}, we obtain in Corollary~\ref{formula} an explicit formula for the graded Betti numbers of a $t$-spread strongly stable ideal.

As for ordinary strongly stable ideals, one defines Borel generators of a $t$-spread strongly stable ideal $I$ as a set of $t$-spread monomials in $I$ with the property that  $I$ is the smallest $t$-spread strongly stable ideal containing these generators. Of particular interest is the case when $I$ has precisely one Borel generator. In the special case when the Borel generator in $K[x_1, \ldots, x_n]$ is  $u=x_{n-t(d-1)}\cdots x_{n-t} x_n,$  the resulting ideal is called a \emph{$t$-spread Veronese ideal}. It is generated by all $t$-spread monomial of degree $\deg(u)$. Theorem~\ref{big} lists the homological and algebraic properties of $t$-spread Veronese ideals and their Alexander duals. The results of this theorem are then used in Theorem~\ref{height} to determine the height of any $t$-spread strongly stable ideal. As a consequence, Cohen-Macaulay $t$-spread strongly stable ideals can be  classified; see Corollary~\ref{cmclassified}.

In Section~\ref{alg}, we study the toric $K$-algebras whose generators are the generators of a $t$-spread principal Borel ideal. Generalizing a result of De Negri \cite{N}, we show that these algebras are Koszul, Cohen-Macaulay normal domains. Finally, in Section~\ref{genini}, we show that the generic initial ideal of a $t$-spread strongly stable ideal is simply obtained by the inverse of the  $t$-fold iterated operator $\sigma$.

It should be noted that the graded Betti numbers of $I$ and $I^{\sigma}$ may coincide not only for $t$-spread strongly stable ideals. In fact, it can be easily seen that if $I=J^{\sigma^n}$ for a monomial ideal $J \subset K[x_1, \ldots, x_n]$, then for any $t$, $I$ and $I^{\sigma^t}$ have the same graded Betti numbers, because for such ideals the application of the operator $\sigma$ simply amounts to rename the variables. It would be interesting to determine all monomial ideals for which $I$ and $I^{\sigma}$ coincide.

\section{$t$--spread strongly stable ideals}

The section is intended to generalize the concepts of stable and squarefree stable ideals.

Let $K$ be a field and $S=K[x_1,\ldots,x_n]$ the polynomial ring in $n$ variables over $K$. We denote by $\Mon(S)$ the set of all the monomials in $S.$
 For a monomial $u$ we denote by $\max(u)$ ($\min(u)$) the maximal (minimal) index $i$ for which $x_i$ divides $u.$

\begin{Definition}
A  $t$-spread monomial ideal $I \subset S$ is called {\em $t$-spread stable},  if for all $t$-spread monomials $u\in I$ and for all $i < \max(u)$ such that $x_i(u/x_{\max(u)})$ is a $t$-spread monomial,  it follows that  $x_i(u/x_{\max(u)})\in I$.

 The ideal $I$  is called {\em $t$-spread strongly stable}, if for all $t$-spread monomials $u\in I$, all $j\in \supp(u)$ and all $i<j$ such that $x_i(u/x_{j})$ is $t$-spread, it follows that  $x_i(u/x_j)\in I$.
\end{Definition}
Note that a $t$-spread strongly stable ideal is also $t$-spread stable.

\begin{Lemma}\label{gen}
Let $I$ be a $t$-spread monomial ideal. The following conditions are equivalent:
\begin{enumerate}
\item[{\em (a)}]  $I$ is $t$-spread strongly stable.
\item[{\em (b)}] If $u \in G(I)$, $j\in \supp(u)$ and $i< j$ such that $x_i(u/x_j)$ is a $t$-spread monomial, then $x_i(u/x_j)\in I$.
\end{enumerate}
\end{Lemma}

\begin{proof}
(a) $\Rightarrow$ (b) is obvious. To prove (b) $\Rightarrow$ (a), let $u \in I$ be a $t$-spread  monomial and $i<j$ such that $u'=x_i(u/x_j)$ is a $t$-spread monomial.  Let $v \in G(I)$ such that $v|u$. If $x_j \notin \supp(v)$, then $v|u'$ and $u' \in I$. Otherwise, if $x_j \in \supp(v)$, then $v'=x_i(v/x_j) \in I$ by our assumption and $v'|u'$ and again we have $u' \in I$.  
\end{proof}


\medskip
The following lemma is crucial for the study of $t$-spread strongly stable ideals.
\begin{Lemma}\label{canonical}
Let $I$ be  a $t$-spread strongly stable ideal and $w \in I$ be a $t$-spread monomial. Then $w=w_1w_2$ such that $\max(w_1) < \min(w_2)$ for some $w_1 \in G(I)$ and $w_2 \in \Mon(S)$.
\end{Lemma}
\begin{proof}
We may assume that $t >0$, because for $t=0$ such a decomposition for $w$ is known; see \cite[Lemma 1.1]{EK}.

Now, let $w=w'_1w'_2$ with $w'_1 \in G(I)$ such that if some $v \in G(I)$ with $v | w$, then $\deg(w'_1) \leq \deg(v)$. Of course, both $w'_1$ and $ w'_2$ are $t$-spread monomials. Suppose that $k=\max(w'_1) - \min(w'_2) \geq 0$. Then we show that there exists $w''_1 \in G(I)$ such that $w=w''_1 w''_2$ for some monomial $w''_2 \in \Mon(S)$ such that $\max(w''_1) - \min(w''_2) < k$.

Let  $j=\max(w'_1)$ and $i= \min(w'_2)$. Then $w''_1=x_i(w'_1/x_j)$ is $t$-spread because $\supp(w''_1) \subseteq \supp(w)$. Let $w''_2=x_j(w'_2/x_i) $. Then $w''_2$ is $t$-spread  as well and  $w=w''_1w''_2$.
Since $I$ is $t$-spread strongly stable and $i<j$, we have $w''_1 \in I$.  Also, $\deg(w'_1)= \deg(w''_1)$, and hence, by the assumption on the $\deg(w'_1)$,   we see that $w''_1 \in G(I)$. Moreover, $\max(w''_1) < \max(w'_1)$ and $\min(w'_2) < \min(w''_2)$ and $\max(w''_1) - \min(w''_2) <k $. By applying induction on $k$, we get the desired result.
\end{proof}

\medskip
A $t$-spread stable ideal need not to have linear quotients. For example, the ideal $I=(x_1x_3x_5,x_1x_4x_6)$  is $2$-spread stable, but does not have linear quotients. However, we have

\begin{Theorem}
\label{ayesha}
The $t$-spread strongly  stable ideals have linear quotients. In particular, they are componentwise linear.
\end{Theorem}

\begin{proof} 
Let $G(I)=\{u_1, u_2, \ldots, u_m\}$ ordered with respect to the pure lexicographical order. Let $r \leq m$ and $J=( u_1,\ldots,u_{r-1})$. Then in order to show that $J:u_r$ is generated by variables, it is enough to show that  for all $1 \leq k \leq r-1$ there exists $x_i \in J:u_r$ such that $x_i$ divides $u_k/\gcd(u_k, u_r)$. Let $u_k=x_{i_1}x_{i_2}\cdots x_{i_s}$ with $i_1\leq i_2\leq \cdots \leq i_s$ and $u_r=x_{j_1}x_{j_2}\cdots x_{j_t}$ with $j_1\leq j_2\leq \cdots \leq j_t$. Since $u_k >_{\lex} u_r$, there exists $d$ with $1 \leq d \leq t$ such that $i_1=j_1, \ldots,  i_{d-1}=j_{d-1}$ and $i_d < j_d$. Let $v=x_{i_d} (u_r/x_{j_d})$. Then  $v=x_{j_1}x_{j_2}\cdots x_{j_{d-1}}x_{i_d}x_{j_{d+1}}\cdots  x_{j_{t}}$. Since  $i_d -j_{d-1} = i_d-i_{d-1} \geq t$ and  $j_{d+1} -i_{d} > j_{d+1}-j_{d} \geq t$, it follows that $v$ is $t$-spread, and so $v \in I$ and $v >_{\lex} u_r$. In fact,  $v \in J$. Indeed, by Lemma~\ref{canonical}, there exists $u_l \in G(I)$ such that $v =u_lw$ and $\max(u_l) < \min(w)$. Suppose that $v \notin J$. Then $u_l \leq_{lex} u_r$. From the presentation of $v=u_lw$, it follows that $v \leq_{lex} u_r$, a contradiction.

Now, as we know that $v \in J$, it follows that $x_{i_d} \in J:u_r$. This  completes the proof, since $x_{i_d}$ divides $u_k/\gcd(u_k, u_r)$.
\end{proof}

Let $I$ be a $t$-spread strongly stable ideal with $G(I)=\{u_1, u_2, \ldots, u_m\}$ ordered  with respect to the pure lexicographic order. As in \cite{HT} we define
\[
\set(u_k) = \{i\:\; x_i\in (u_1,\ldots,u_{k-1}):u_k\} \quad \text{for $k=1,\ldots,m$}.
\]
The proof of Theorem~\ref{ayesha} shows that $\set(u_k)$ is the set of positive integers $i$ satisfying
\begin{equation}\label{henning}
\text{$i<\max(u_k)$, $i\not\in\supp(u_k)$  and $i-j\geq t$ for all $j\in \supp(u_k)$ with $j<i$}.
\end{equation}

We  set  $I_j=(u_1,\ldots,u_j)$ for $j=1,\ldots, m$.
Let $M(I)$ be the set of all monomials in $I$.
The  {\em decomposition map} $g\: M(I)\to G(I)$ is defined as follows: for  $u\in M(I)$ we let   $g(u)=u_j$, where  $j$
is the smallest number such that $u\in I_j$. The decomposition map is {\em regular}, if
$\set(g(x_iu_k))\subset\set(u_k)$ for all $i\in\set(u_k)$ and all $u_k\in G(I)$.

\medskip
The resolution of monomial ideals with linear quotients and regular decomposition function  can be explicitly described; see \cite[Theorem 1.12]{HT}. Stable and squarefree stable ideals have regular decomposition functions. However, even 2-spread  strongly stable monomial ideals in general do not have regular decomposition functions.

For example, consider the 2-spread  strongly stable ideal
\[
I=(x_1x_3,x_1x_4, x_1x_5,x_1x_6,x_2x_4, x_2x_5,x_2x_6,x_3x_5,x_3x_6).
\]
 Then, $\set(x_3x_6)= \{1,2,5\}$,  $g(x_2x_3x_6)= x_2x_6$ and $\set(g(x_2x_3x_6))=\{1,4,5\} \not\subseteq \set(x_3x_6)$.

 \medskip
In what follows, we will establish a bijection between $t$-spread strongly stable ideals and $t+1$-spread strongly stable ideals which preserves the graded Betti numbers.

Let $T= K[x_1, x_2, \ldots]$ be a polynomial ring in infinitely many variables. We denote by  $\Mon(T;t)$ the set of all $t$-spread monomials in $T$. Then $\Mon(T;0)$ is just the set of all monomials of $T$ which we simply denote by $\Mon(T)$.

\begin{Definition}
Let $u=\prod_{j=1}^{d}x_{i_j} \in T$ with $i_1 \leq i_2 \leq \cdots \leq i_d$. Then we define
$\sigma: \Mon(T) \rightarrow \Mon(T)$ by
\[
\sigma(u)=\prod_{j=1}^{d} x_{i_{j}+(j-1)}.
\]
\end{Definition}

Note that, $\sigma$ induces a map  $\Mon(T;t) \rightarrow \Mon(T;t+1)$ which we again denote by $\sigma$. Indeed, if $u$ is a $t$-spread monomial then $\sigma(u)$ is a $t+1$-spread monomial, because $(i_{j+1}+j)- (i_j+(j-1))= i_{j+1}- i_j+1 \geq t+1$.
\begin{Lemma}\label{bijection}
The map $\sigma: \Mon(T;t) \rightarrow \Mon(T;t+1)$ is bijective.
\end{Lemma}

\begin{proof}
 Let $u=\prod_{j=1}^{d}x_{i_j} \in T$ with $i_1 \leq i_2 \leq \cdots \leq i_d$. We define the inverse map of $\sigma$ by $\tau:\Mon(T;t+1) \rightarrow \Mon(T;t)$ by
\[
\tau(u)=\prod_{j=1}^{d} x_{i_{j-(j-1)}}.
\]
\end{proof}

\begin{Corollary}\label{bijection}
The iterated map $\sigma^t: \Mon(T) \rightarrow \Mon(T;t)$ establishes a bijection between the set of all monomial in $T$ and the set of all $t$-spread monomials in $T$.
\end{Corollary}

\begin{Definition}
Let $I$ be a monomial ideal. Then we let $I^{\sigma}$ be the ideal generated by the monomials $\sigma(u)$ with $u \in G(I)$.
\end{Definition}
Observe that if $I$ is a $t$-spread ideal then $I^{\sigma}$ is a $t+1$-spread ideal.

\begin{Proposition}\label{equal}
Let $I$ be a monomial ideal. Then $I$ is a  $t$-spread strongly stable ideal if and only if $I^{\sigma}$ is a $t+1$-spread strongly stable ideal.
\end{Proposition}

\begin{proof}
Let $I$ be a $t$-spread ideal and
 $\sigma(u)=\prod_{j=1}^{d} x_{i_{j}+(j-1)}$ with $u \in G(I)$. We want to show that for all $j \in \supp(\sigma(u))$ and $k<j$ such that $v=x_k(\sigma(u)/x_j)$ is a $t+1$-spread monomial then $v \in I^{\sigma}$. Since $x_j | \sigma(u)$, it follows that $j=i_l+{l-1}$ for some $1 \leq l \leq d$. Then
\[
v=x_{i_1}x_{i_2+{1}} \cdots x_{{i_{l-1}}+(l-2)} x_k x_{{i_{l+1}}+l}\cdots x_{{i_d}+(d-1)}
\]
Let $w=\tau(v)$. Then, first we show that $w \in I$. Indeed,
\[
w=x_{i_1}x_{i_2} \cdots x_{{i_{l-1}}} x_{k-(l-1)} x_{{i_{l+1}}}\cdots x_{{i_d}},
\]
therefore, $w=x_{k-(l-1)}(u/x_{i_l})$ and $k-(l-1) <i_l$ . Moreover, $w$ is $t$-spread. Indeed, since $v$ is $t+1$-spread, we have $k-(i_{l-1}+(l-2) )\geq t+1$ which implies $k-(l-1)-i_{l-1}\geq t$, and we have $i_{l+1}+l-k\geq t+1$ which implies $i_{l+1}-(k-(l-1)) \geq t+1$. Then, by Lemma~\ref{canonical}, $w=w_1w_2$ such that $\max(w_1) < \min(w_2)$. This implies that $v=\sigma(w)=\sigma(w_1)w'$ where $w'$ is a monomial. Therefore, $v \in I^{\sigma}$.

The converse may be handled in a similar way.
\end{proof}

For the proof of Theorem~\ref{betti}, we use the following result which is the immediate consequence of  \cite[Lemma 1.5]{HT}.


\begin{Lemma}\label{set1}
Let $I$ be a monomial ideal with linear quotients. Then
\[
\beta_{i,i+j}(I)= |\{\alpha \subset \set (u)\; : \; u \in G(I)_j \text{ and } |\alpha|=i\}|,
\] where $G(I)_j=\{u \in G(I)\; : \; \deg(u)=j\}$.
\end{Lemma}

\begin{Theorem}\label{betti}
Let $I$ be a $t$-spread strongly stable ideal. Then $\beta_{i,i+j}(I)= \beta_{i,i+j}(I^{\sigma})$ for all $i$ and $j$.
\end{Theorem}

\begin{proof}
Let $u=x_{i_1}x_{i_2}\cdots x_{i_d} \in G(I)$.  Let $\set(u)=\{a_1< \cdots < a_r\}$ and
\[
b_i=a_i+\max\{l \; : \; i_l <a_i\}
\]
for $i=1, \ldots, l$.

 We claim that $b_1 < \cdots < b_r$ and $\set(\sigma(u))=\{b_1, \ldots , b_r\}$. The claim together with Lemma~\ref{set1} yields the desired result.

\medskip
Proof of the claim: Let $k <j$ and $i_l < a_k<i_{l+1}$ and $i_m < a_j<i_{m+1}$. Then $m \geq l$ and $b_j-b_k= a_j+m -(a_k+l)= a_j-a_k+(m-l) >0$.

Next, we show that $b_i \in \set(\sigma(u))$. Indeed, if $a_i \in \set(u)$ and $i_l < a_i < i_{l+1}$, then by (\ref{henning}) we have $a_i -i_l \geq t$. Therefore,  $i_{l}+(l-1) < a_i+l < i_l+(l+1)$ and $a_i+l - (i_l+(l-1)) \geq t+1$. Since, $b_i=a_i+l$, this shows that $b_i \in \set(\sigma(u))$.

Conversely, let $c \in \set(\sigma(u))$. Then from (\ref{henning}), we see that there exists an integer $l$ such that $i_l +(l-1)< c < i_{l+1}+l$ and $c-(i_l+(l-1)) \geq t+1$. This shows that $i_l < c -l< i_{l+1}$ and $(c-l)-i_l \geq t$. Therefore, $c-l =a_i$ for some $i$ and $c=a_i+l=b_i$.
\end{proof}

In general, Theorem~\ref{betti} is not valid for an arbitrary $t$-spread monomial ideal. For example, let $I=(x_1^2, x_2^2)$. Then $I^{\sigma}= (x_1x_2,x_2x_3)$, and $I^{\sigma}$ has a linear resolution while $I$ does not.

\begin{Corollary}\label{formula}
Let $I$ be a $t$-spread strongly stable ideal. Then
\[
\beta_{i,i+j}(I) = \sum_{u \in G(I)_j} \binom{ \max(u)-t(j-1)-1}{i}.
\]
\end{Corollary}
\begin{proof}
We know that $I^{\tau^t}$ is strongly stable. From \cite{EK}, we know that
\[
\beta_{i, i+j}(I^{\tau^t})=\sum_{u \in G(I^{\tau^t})_j} \binom{\max(u)-1}{i}.
\]
By Theorem~\ref{betti}, we have $\beta_{i, i+j}(I^{\tau^t})= \beta_{i, i+j}(I)$, therefore,
\[
\beta_{i, i+j}(I)=\sum_{u \in G(I)_j} \binom{\max(\tau^t(u))-1}{i}.
\]
The proof follows, because $\max(\tau^t(u)))= \max(u)-t(\deg(u)-1)$, for all $u \in G(I)$.
\end{proof}


\section{$t$--spread Borel generators}

In the theory of stable ideals, Borel generators play an important role. In this section, we introduce the similar concept of $t$-spread strongly stable ideals. 

Let $u_1,\ldots, u_m$ be $t$-spread monomials in $S$. There 
 exists a unique smallest $t$-spread strongly stable ideal containing $u_1,\ldots, u_m$, which we denote by $B_t(u_1, \ldots, u_m)$. The monomials $u_1, \ldots, u_m$ are called the {\em $t$-spread Borel generators} of $B_t(u_1, \ldots, u_m)$.

For example, let $I=B_2(x_2x_4, x_1x_5)$. Then $G(I)=\{x_1x_3, x_1x_4,x_1x_5,x_2x_4\}$.

\begin{Proposition}\label{gen}
Let $I=B_t(u_1, \ldots, u_m).$ Then $I^{\sigma} = B_{t+1}(\sigma(u_1), \ldots, \sigma(u_m))$.
\end{Proposition}

\begin{proof}
Let $w \in G(I)$ and $w=x_{j_1}\cdots x_{j_d}$. Then there exists $u_l=x_{i_1}\cdots x_{i_d}$ such that $j_k \leq i_k$ for all $k=1, \ldots, d$. It gives $j_k + (k-1)\leq i_k+(k-1)$ for all $k=1, \ldots, d$. Therefore, $\sigma(w) \in B_{t+1}(\sigma(u_l)) \subseteq B_{t+1}(\sigma(u_1), \ldots, \sigma(u_m))$. Since $I^{\sigma}$ is generated by elements $\sigma(w)$ with $w \in G(I)$, it shows that $I^{\sigma} \subseteq B_{t+1}(\sigma(u_1), \ldots, \sigma(u_m))$. Furthermore,  $B_{t+1}(\sigma(u_1), \ldots, \sigma(u_m))$ is the smallest $t$-spread strongly stable containing $\sigma(u_1), \ldots, \sigma(u_m)$. Therefore, $B_{t+1}(\sigma(u_1), \ldots, \sigma(u_m)) \subseteq I^{\sigma} $, because $\sigma(u_1), \ldots, \sigma(u_m) \in I^{\sigma}$ and $I^{\sigma}$ is $t$-spread strongly stable .
\end{proof}
We call a $t$-spread strongly stable ideal $I$ {\em $t$-spread principal Borel}, if  there exists a $t$-spread monomial $u \in I$ such that $I=B_t(u)$.

\medskip
Let $u=x_{i_1}\cdots x_{i_d}$. Observe that $x_{j_1}\cdots x_{j_d} \in G(B_t(u))$ if and only if
\begin{enumerate}
\item[(i)] $j_1 \leq i_1, \ldots, j_d \leq i_d$, and
\item[(ii)] $j_k-j_{k-1} \geq t$ for $k=2, \ldots, d$.
\end{enumerate}

\medskip
In what follows, we study an important special class of $t$-spread principal Borel ideals.
\begin{Definition}
Let $d\geq 1$  be an integer. A monomial ideal in $S=K[x_1, \ldots, x_n]$  is called a {\em $t$-spread Veronese  ideal of degree $d$},  if it is generated by all $t$-spread monomials of degree $d$. 
\end{Definition}

We denote by $I_{n,d,t} \subset S$, the $t$-spread Veronese  ideal in $S$ generated in degree $d$. Note that $I_{n,d,t} \neq (0)$ if and only if $n > t(d-1)$. Observe that the $t$-spread Veronese ideal of degree $d$ is indeed a $t$-spread principal Borel ideal. In fact,
\[
I_{n,d,t}=B_t(\prod_{i=0}^{d-1}x_{n-it}).
\]

Theorem~\ref{gen} implies that
\[
I_{n,d,t}^{\sigma}=I_{n+d-1,d,t+1} \quad \text{and} \quad I_{n,d,t}^{\tau}=I_{n-d+1,d,t-1} \quad \text{if } t \geq 1.
\]

\medskip
Therefore,
\begin{equation}\label{id}
[(x_1,\ldots, x_{n-t(d-1)})^d]^{\sigma^t}= I_{n,d,t}.
\end{equation}

\medskip
There exists a simplicial complex $\Delta$ on the
vertex set $[n]$ such that $I_{n,d,t}$ is the Stanley-Resner ideal of $\Delta.$ We denote by $I_{n,d,t}^\vee$ the Stanley-Reisner ideal of
the Alexander dual of $\Delta.$

\begin{Theorem}\label{big}
 Let $t\geq 1$ be an integer and  $I_{n,d,t} \subset S$  the $t$-spread Veronese  ideal generated in degree $d$. We assume that 
$\bigcup_{u\in G(I_{n,d,t})}\supp(u)=\{x_1,\ldots,x_n\}.$ Then we have the following:
\begin{enumerate}
\item[{\em (a)}] $\height (I_{n,d,t})=n-t(d-1)$.
\item[{\em (b)}] $I_{n,d,t}^\vee$ is generated by the monomials
\[
\prod_{i=1}^{n} x_i/(v_{i_1,t}\cdots v_{i_{d-1},t}) \text{ with } i_{j+1}-i_j \geq t, \text{ for }\quad 1 \leq j \leq d-2,
\]
where $v_{i_k,t}=x_{i_k}x_{{i_k}+1}\cdots x_{{i_k}+t-1}$ for $1 \leq k \leq d-1.$
\item[{\em (c)}] $I_{n,d,t}$ is Cohen-Macaulay and has a linear resolution.
\item[{\em (d)}]  $\beta_i (S/I_{n,d,t})=\binom{d+i-2}{d-1}\binom{n-(t-1)(d-1)}{d+i-1}$ for all $i\geq 1.$ In particular,
$\mu (I_{n,d,t})={\binom{n-(t-1)(d-1)}{d}}.$
\item[{\em (e)}]  $\beta_i (S/I_{n,d,t}^\vee)=\binom{n-t(d-1)+i-1}{i-1}\binom{n-t(d-1)+d}{d-i}$ for all $i\geq 1.$ In particular,
$\mu (I_{n,d,t}^\vee)={\binom{n-t(d-1)+d}{d-1}}.$
\end{enumerate}
\end{Theorem}

\begin{proof}
Let $\Delta$ be the simplicial complex whose Stanley-Reisner ideal is $I_{n,d,t}$ and let $\MF(\Delta)$ the set of facets of $\Delta.$
We prove that every facet of $\Delta$ is of the form
{\small
\[
F=\{j_1,j_1+1,\ldots,j_1+(t-1),j_2,j_2+1,\ldots,j_2+(t-1),\ldots, j_{d-1},j_{d-1}+1,\ldots,j_{d-1}+(t-1)\}
\]}
 for some  $j_1,\ldots,j_{d-1}$  such that $j_\ell-j_{\ell-1}\geq t$ for $2\leq \ell\leq d-1.$

This shows that all the facets of $\Delta$ have the same cardinality, namely $t(d-1),$ thus  $\dim \Delta=t(d-1)-1.$ It follows  that $\dim(S/I_{n,d,t})=t(d-1),$ thus $\height (I_{n,d,t})=n-t(d-1)$ which proves (a). Moreover,  $I_{n,d,t}$ has the primary decomposition
\[I_{n,d,t}=\bigcap_{F\in \MF(\Delta)}P_{[n]\setminus F}\] where $P_{[n]\setminus F}$ is the monomial prime ideal generated by all the variables $x_j$ with $j\in [n]\setminus F.$  By \cite[Corollary 1.5.5]{HHBook}, statement (b) holds.

To begin with, we show that every set
\begin{equation}\label{eq1}
F=\{j_1,\ldots,j_1+(t-1),j_2,\ldots,j_2+(t-1),\ldots, j_{d-1},\ldots,j_{d-1}+(t-1)\}
\end{equation}
 for some  $j_1,\ldots,j_{d-1}$  such that $j_\ell-j_{\ell-1}\geq t$ for $2\leq \ell\leq d-1$ is a facet of $\Delta.$ We have
$F\in \Delta$ since $x_F=\prod_{j\in F}x_j\not\in I_{\Delta}$. On the other hand, we claim that $F\cup\{j\}\not\in \Delta$ for every
$j\in [n]\setminus F.$ This will show that $F$ is indeed a facet of $\Delta.$

Let $j\in [n]\setminus F.$ If $j<j_1,$ we get
\[x_jx_{j_1+(t-1)}\cdots x_{j_{d-1}+(t-1)}\in I_{\Delta},\]
thus $\{j, j_1+(t-1),\ldots, j_{d-1}+(t-1)\}$ is a non-face of $\Delta,$ which implies that $F\cup\{j\}\not\in \Delta.$ If
$j\geq j_{d-1}+t,$ we get the non-face $\{j_1,\ldots,j_{d-1},j\}$, thus $F\cup\{j\}\not\in \Delta.$ Finally, if there exists
$2\leq\ell\leq d-1$ such that $j_{\ell-1}+(t-1)<j<j_\ell,$ then $\{j_1,\ldots,j_{\ell-1},j,j_{\ell}+(t-1),\ldots, j_{d-1}+(t-1)\}$ is a non-face of $\Delta.$ Consequently, $F\cup\{j\}\not\in \Delta.$

Therefore, we have shown that every set $F$ as in (\ref{eq1}) is a facet of $\Delta.$

Our purpose is to show that the sets of the form  (\ref{eq1})  are the only facets. This is equivalent to showing that for every face $G\in \Delta$, there exists
$F\in \MF(\Delta)$ of the form (\ref{eq1}) which contains $G.$

Let $G\in \Delta$ and $i_1=\min G.$ Inductively, for  $\ell\geq 2,$ we set \[i_\ell=\min\{i\in G: i\geq i_{\ell-1}+t\}.\]
The sequence $i_1<i_2<\cdots$ has at most $d-1$ elements. Otherwise, $G\supseteq \{i_1,\ldots,i_d\}$ with
$i_\ell\geq i_{\ell-1}+t$ for $2\leq \ell\leq d.$ But $\{i_1,\ldots,i_d\}\not\in \Delta$ since
$x_{i_1}\cdots x_{i_d}\in I_{\Delta}.$ Thus $G\not\in\Delta,$ a contradiction. Therefore, $G$ has the form
\[G=\{i_1,i_1+1,\ldots,i_1+q_1,\ldots, i_k,i_k+1,\ldots,i_k+q_k\}\]
for some $k\leq d-1, 0\leq q_1,\ldots,q_k\leq t-1,$ and $i_\ell\geq i_{\ell-1}+t$ for $2\leq \ell\leq k.$ Obviously,
$G\subseteq G^\prime$ where
\[G^\prime=\{i_1,i_1+1,\ldots,i_1+(t-1),\ldots,i_{k-1},i_{k-1}+1,\ldots,i_{k-1}+(t-1),i_k,\ldots,i_k+q\}\]
where we set $q=q_k.$

\textbf{Claim}. For $k\leq d-2,$ there exists $H\in \Delta,$ $H\supset G^\prime\supset G,$ with
\[H=\{i^\prime_1,\ldots,i^\prime_1+(t-1),\ldots,i^\prime_{k},i^\prime_{k}+1,\ldots,i^\prime_{k}+(t-1),i^\prime_{k+1},\ldots, i^\prime_{k+1}+q^\prime\}
\] for some $0\leq q^\prime\leq t-1,$ $i^\prime_1\leq t,$ and $i^\prime_{\ell}\geq i^\prime_{\ell-1}+t$ for $2\leq \ell\leq k+1.$

\emph{Proof of the Claim.} If $i_1=\min G^\prime\geq t+1,$ then $G\subset G^\prime\subset H=\{1,\ldots,t\}\cup G^\prime\in \Delta$ and the claim follows.

Let now $i_1\leq t$ and assume for the beginning that $i_\ell=i_{\ell-1}+t$ for $2\leq \ell\leq k.$ Then
\[i_k=i_1+(k-1)t\leq kt\leq (d-2)t\leq n-t-1.\] In the last inequality we used the  condition $n\geq 1+(d-1)t$ which must be satisfied by
$n.$ Then we get $i_k+t\leq n-1$ hence we may take
\[H=\{i_1,i_1+1,\ldots,i_1+(t-1),\ldots,i_{k},i_{k}+1,\ldots,i_{k}+(t-1),i_{k+1}=i_k+t\}.\]
To complete the proof of the Claim, we need to consider a last case, namely when there exists $\nu$ such that $i_\nu>i_{\nu-1}+t.$
Let $\ell=\max\{\nu: i_\nu>i_{\nu-1}+t\}.$ Then it follows that $i_k>i_{\ell-1}+(k-\ell+1)t$ and we may take
\[H=\{i_1,\ldots,i_1+(t-1),\ldots,i_{\ell-1},\ldots,i_{\ell-1}+(t-1),i^\prime_\ell,\ldots, i^\prime_\ell+(t-1),\ldots,\]
\[i^\prime_k,\ldots, i^\prime_k+(t-1),i^\prime_{k+1}\ldots,i^\prime_{k+1}+s^\prime\}\]
for some $s^\prime\geq 0,$ where $i^\prime_\ell=i_{\ell-1}+t,i^\prime_{\ell+1}=i_{\ell-1}+2t,\ldots, i^\prime_{k+1}=i_{\ell-1}+(k-\ell+1)t.$

By our Claim, it is now clear that every face $G\in \Delta$ is contained in a larger face $H$ of the form
\begin{equation}\label{eq2}
H=\{i_1,\ldots,i_1+(t-1),\ldots,i_{d-2},\ldots,i_{d-2}+(t-1),i_{d-1},\ldots,i_{d-1}+s\}
\end{equation}
for some $0\leq s\leq t-1,$  where $i_1\leq t, $ and $i_\ell\geq i_{\ell-1}+t$ for $2\leq \ell\leq d-1.$

It remains to show that there exists  $F\in \MF(\Delta)$ which contains $H.$ But this follows if we show that for every $s\leq t-2,$
$H$ is contained in a face of $\Delta$ of the form
\[\{i^\prime_1,\ldots,i^\prime_1+(t-1),\ldots,i^\prime_{d-2},\ldots,i^\prime_{d-2}+(t-1),i^\prime_{d-1},\ldots, i^\prime_{d-1}+(s+1)\}.\]

Let $s\leq t-2.$ Of course, if $i_{d-1}+s<n,$ then we may get the larger face immediately, just by adding to $H$ the vertex $i_{d-1}+(s+1).$
Let $i_{d-1}+s=n.$ If $i_\ell=i_{\ell-1}+t$ for all $2\leq \ell\leq d-1,$ then $i_{d-1}=i_1+(d-2)t,$ thus $i_1+(d-2)t+s=n$ which
implies that
\[i_1=n-(d-2)t-s\geq 1+(d-1)t-(d-2)t-s\geq 3.\] Then, we can take
\[H\subset \{i^\prime_1,\ldots,i^\prime_1+(t-1),\ldots,i^\prime_{d-2},\ldots,i^\prime_{d-2}+(t-1),i^\prime_{d-1},\ldots, i^\prime_{d-1}+(s+1)\}\] where $i^\prime_1=i_1-1,i^\prime_2=i_2-1,\ldots,i^\prime_k=i_k-1.$

Finally, let us choose the maximal $\ell$ such that $i_\ell>i_{\ell-1}+t.$ 
In this case, we take
\[H\subset \{i^\prime_1,\ldots,i^\prime_1+(t-1),\ldots,i^\prime_{d-2},\ldots,i^\prime_{d-2}+(t-1),i^\prime_{d-1},\ldots, i^\prime_{d-1}+(s+1)\}\] with $i^\prime_1=i_1,\ldots, i^\prime_{\ell-1}=i_{\ell-1}, i^\prime_\ell=i_{\ell}-1, i^\prime_{\ell+1}=i_{\ell+1}-1,\ldots,
i^\prime_{d-1}=i_{d-1}-1.$
\medskip

In order to prove
 that $I_{n,d,t}$ has a linear resolution, it is enough to apply Theorem~\ref{ayesha}. Since $I_{n,d,t}$ is generated in a single degree, it follows that it has a linear resolution.

Next, we show that $I_{n,d,t}^\vee$ has linear quotients. Then, by \cite[Proposition 8.2.5]{HHBook}, it follows that the simplicial complex
$\Delta$ is shellable, thus, by \cite[Theorem 8.2.6]{HHBook}, $I_\Delta=I_{n,d,t}$ is Cohen-Macaulay.

Let $w_1,\ldots,w_q$ be the minimal monomial generators of $I_{n,d,t}^\vee$ ordered decreasingly with respect to the lexicographic order. Let
\[w_i=\prod_{i=1}^{n} x_i/(v_{i_1}\cdots v_{i_{d-1}}) \text{ and } w_j=\prod_{i=1}^{n} x_i/(v_{j_1}\cdots v_{j_{d-1}})\]
with $i\neq j.$ In order to simplify a little the notation, we removed the index $t$ in $v_{j_k,t}$ and $v_{i_k,t}.$ A simple calculation shows that
\[ \frac{w_i}{\gcd(w_i,w_j)}=\frac{v_{j_1}\cdots v_{j_{d-1}}}{\gcd(v_{i_1}\cdots v_{i_{d-1}},v_{j_1}\cdots v_{j_{d-1}})}.\]

Let $i<j.$ Then $w_i>_{\lex} w_j,$ that is, $v_{j_1}\cdots v_{j_{d-1}}>_{\lex} v_{i_1}\cdots v_{i_{d-1}}$ which is equivalent to the condition
that there exists an integer $s\geq 1$ such that $j_1=i_1,\ldots,j_{s-1}=i_{s-1}$ and $j_s<i_s.$

We first observe that $x_{j_s}\mid (w_i/\gcd(w_i,w_j))$  since $x_{j_s}\mid v_{j_1}\cdots v_{j_{d-1}}$ and it does not divide the product
$v_{i_1}\cdots v_{i_{d-1}}$ because $i_s>j_s.$

Let us assume that there exists a  least integer $\ell\leq d-2$ such that $j_{\ell+1}>j_\ell+t.$
Let  \[w_k=\prod_{i=1}^{n} x_i/(v_{j_1}\cdots v_{j_{s-1}}v_{j_s+1}v_{j_s+2}\cdots v_{j_{\ell}+1}v_{j_{\ell+1}}
\cdots v_{j_{d-1}}).\]

Obviously, $w_k>_{\lex} w_j$, thus $k<j,$  and we claim that $w_k/\gcd(w_k,w_j)=x_{j_s.}$ An easy calculation shows
that
\[\gcd(v_{j_1}\cdots v_{j_{s-1}}v_{j_s+1}v_{j_s+2}\cdots v_{j_{\ell}+1}v_{j_{\ell+1}}\cdots v_{j_{d-1}},v_{j_1}\cdots v_{j_{d-1}} )=
\frac{v_{j_1}\cdots v_{j_{d-1}}}{x_{j_s}}.\] Then,
\[\frac{w_k}{\gcd(w_k,w_j)}=\frac{v_{j_1}\cdots v_{j_{d-1}}}{((v_{j_1}\cdots v_{j_{d-1}})/x_{j_s})}=x_{j_s}.\]

If $j_{\ell+1}=j_\ell+t$ for $2\leq \ell\leq d-2,$ we get \[j_{d-1}=j_s+(d-s-1)t<i_s+(d-s-1)t\leq i_{d-1}\leq n-t+1.\] Thus
$j_{d-1}+(t-1)\leq n$, and we may consider the monomial
$v_{j_{d-1}+1}$.
In this case we  take \[w_k=\prod_{i=1}^{n} x_i/(v_{j_1}\cdots v_{j_{s-1}}v_{j_s+1}v_{j_s+2}\cdots v_{j_{d-1}+1})\] and check that
$w_k/\gcd(w_k,w_j)=x_{j_s}$.

Finally, for the calculation of the Betti numbers of $I_{n,d,t}$ and $I_{n,d,t}^\vee$, we employ \cite[Theorem 4.5]{BH98} which gives
the Betti numbers of a Cohen-Macaulay  ideal $I$ in a polynomial ring $R$ with pure resolution of type $(d_1,\ldots,d_p)$. We have
\[\beta_i(R/I)=(-1)^{i+1}\prod_{j\neq i} \frac{d_j}{d_j-d_i}, i\geq 1.\] In our case, the type of the resolution of $S/I_{n,d,t}$ is given by
$d_j=d+j-1$ for $1\leq j\leq p=n-t(d-1).$ Therefore,
\[\beta_i(S/I_{n,d,t})=(-1)^{i+1}\prod_{j=1}^{i-1}\frac{d+j-1}{j-i}\prod_{j=i+1}^p\frac{d+j-1}{j-i}=\] \[
=\frac{d(d+1)\cdots (d+i-2)}{(i-1)!}\cdot \frac{(d+i)(d+i+1)\cdots (d+p-1)}{(p-i)!}=\] \[=\binom{d+i-2}{d-1}\binom{n-(d-1)(t-1)}{d+i-1}.\]

By Eagon-Reiner Theorem \cite[Theorem 8.1.9]{HHBook}, it follows that $I_{n,d,t}^\vee$ is also Cohen-Macaulay and has a linear resolution. Thus, we may compute the Betti numbers of
$S/I_{n,d,t}^\vee$ as
we did for $S/I_{n,d,t}.$ Note that, in this case, we have $\height(I_{n,d,t}^\vee)=\projdim (S/I_{n,d,t}^\vee)=d$, and the degree of the generators of $I_{n,d,t}^\vee$ is equal to the height of $I_{n,d,t}.$ We omit the remaining part of the calculation of Betti numbers since it is completely similar to the above part of the proof.

We end the proof with the following remark. One may get an alternative proof of part  (d) by using (\ref{id}) and Theorem~\ref{betti}.
 \end{proof}

As an application of Theorem~\ref{big}, we prove the following

\begin{Theorem}\label{height}
Let $I$ be $t$-spread strongly stable ideal. Then
\[
\height(I)=\max \{\min(u): u \in G(I)\} .
\]
\end{Theorem}

\begin{proof}
Let $u_0\in G(I)$ such that $\min(u_0)=\max \{\min(u): u \in G(I)\}$, and let $P=(x_i : i \leq \min(u_0))$. Then $I \subset P$, because  for all $ w \in G(I)$ one has $\min(w) \leq \min(u_0)$.  This shows that $\height(I) \leq \min(u_0)$.

Conversely, let $u_0=x_{i_1}\cdots x_{i_d}$. Then $u'_0= x_{i_1} x_{i_1+t}\cdots x_{i_1+t(d-1)}$ belongs to $I$ because $I$ is $t$-spread strongly stable. Let $I'=B_t(u'_0)$. Then $I' \subset I$ and Theorem~\ref{big} implies that
\[
\height(I) \geq \height(I') = i_1+t(d-1)-t(d-1)=i_1=\min(u'_0)=\min(u_0).
\]
\end{proof}

\begin{Corollary}
\label{cmclassified}
Let $I \subset  S=K[x_1, \ldots, x_n]$ be a $t$ spread strongly stable ideal such that $\bigcup_{u \in G(I)}\supp(u)=\{x_1, \ldots, x_n\}.$ Then $S/I$ is Cohen-Macaulay if and only if there exists $u \in G(I)$ of degree $d$ such that $u=x_{n+t(d-1)}\cdots x_{n-t} x_n $.

In particular, if $I$ is generated in a single degree then $S/I$ is Cohen-Macaulay if and only if $I$ is $t$-spread Veronese.
\end{Corollary}
\begin{proof}
From Corollary~\ref{formula}, it follows that
\[
\pd(S/I)=\max \{ \max(u)-t(\deg(u)-1) \; : u \in G(I)  \},
\]
and from Theorem~\ref{height}, it follows that
\[
\dim(S/I)=n- \max\{\min(u) \; : u \in G(I)  \}.
\]
By using Auslander-Buchsbaum theorem we conclude that $S/I$ is Cohen-Macaulay if and only if
\begin{equation}\label{equ}
\max \{ \max(u)-t(\deg(u)-1) \; : u \in G(I)  \}=\max\{\min(u) \; : u \in G(I)  \}.
\end{equation}
Let $u_0 \in G(I) $ with $\min(u_0)=\max\{\min(u) \; : u \in G(I)  \}$. Since \[\min(u) \leq  \max(u)-t(\deg(u)-1)\] for all $ u \in G(I),$ equality  (\ref{equ}) holds  if  and only if
\[\min(u_0)=  \max(u_0)-t(\deg(u_0)-1).\] 

In other words, $S/I$ is Cohen-Macaulay if and only if there exists $u_0 \in G(I)$ with
\[
\min(u_0)=\max\{\min(u)\; : u \in G(I)\} \text{ and  } u_0=x_{i_1}x_{i_1+t}\cdots x_{i_1+t(d-1)}
\]
Since $\bigcup_{u \in G(I)}\supp(u)=\{x_1, \ldots, x_n\}$, there exists $u \in G(I)$ such that $\max(u)=n$ and $\min (u) \leq \min(u_0)$. Note that
\[
\max(u) -t(\deg(u)-1) \leq i_1= \min(u_0).
\]
Therefore, this implies that $n \leq i_1+t(\deg(u)-1)$. Suppose that $\deg (u) \leq \deg(u_0)=d$, then it follows that $n =i_1+t(d-1)$, as required. On the other hand, if $\deg (u)> d$, then $u=x_{j_1}\cdots x_{j_d}x_{j_{d+1}}\cdots x_n$ with $j_1 \leq j_2 \leq \cdots \leq n.$  Let $u'=x_{j_1}\cdots x_{j_d}$. Since $u \in G(I)$, we have $j_d > i_1+t(d-1)$, otherwise, $j_k \leq i_k$ for all $1 \leq k \leq d$. Then, since $I$ is $t$-strongly stable, we obtain $u' \in I$ and $u'|u$, a contradiction.

Since $j_d > i_1+t(d-1)$, we get
\[
\max(u)-t(\deg(u)-1) \geq \max(u')-t(\deg(u')-1) > i_1,
\]
a contradiction.
\end{proof}

\section{$t$--spread principal Borel algebras}\label{alg}

Let $t\geq 1$ and $u \in S$ be a $t$-spread monomial. In this section, we consider  the toric algebra $K[B_t(u)]$ which is generated by the monomials $v$ with $v \in G(B_t(u))$. If $u=x_{n-t(d-1)}\cdots x_{n-t}x_n$, then $B_t(u)=I_{n,d,t}$ and in this case $K[B_t(u)]$ is called a {\em $t$-spread Veronese algebra}.



Let us first recall the notion of sortable sets of monomials. For more information we refer to \cite[Section 6.2]{EHBook}.
Let $u,v$ two monomials of degree $d.$ We write $uv=x_{i_1}x_{i_2}\cdots x_{i_{2d}}$ with $1\leq i_1\leq i_2\leq \cdots\leq i_{2d}$, and
consider the monomials $u^\prime=x_{i_1}x_{i_3}\cdots x_{i_{2d-1}}, v^\prime =x_{i_{2}}x_{i_4}\cdots x_{i_{2d}}.$ The pair
$(u^\prime,v^\prime)$ is called the \emph{sorting} of $(u,v).$ We write $(u^\prime,v^\prime)=\sort(u,v).$  A subset $\MS\subset S_d$ is called \emph{sortable} if $\sort(u,v)\in \MS\times \MS$ for all $(u,v)\in \MS\times \MS.$

\begin{Proposition}\label{sort}
The set $G(B_t(u))$  is sortable.
\end{Proposition}

\begin{proof}
Let $u=x_{i_1}\cdots x_{i_d}$ with $i_1 \leq \cdots \leq i_d$. Let $w,v\in G(B_t(u))$ and write $wv=x_{j_1}x_{j_2}\cdots x_{j_{2d}}.$ Then $w^\prime=x_{j_1}x_{j_3}\cdots x_{j_{2d-1}}, v^\prime =x_{j_{2}}x_{j_4}\cdots x_{j_{2d}}$. By \cite[Lemma 2.7]{N}, we have $j_{2k}, j_{2k-1} \leq i_k$ for $k=1 \ldots, d$.

It remains to be shown that $j_{2\ell+1}-j_{2\ell-1}\geq t$ and $j_{2\ell+2}-j_{2\ell}\geq t$ for all $1\leq \ell\leq d-1.$ We prove only the first inequality since the second one may be proved in a similar way.
If $x_{j_{2\ell-1}},x_{j_{2\ell+1}}$ divide the same monomial, say $w,$ then the inequality holds since $w\in G(B_t(u)).$ Else, we may consider
that $x_{j_{2\ell-1}}\mid w$ and $x_{j_{2\ell+1}}\mid v.$ If $x_{j_{2\ell}}\mid w,$ then $j_{2\ell+1}-j_{2\ell-1}\geq j_{2\ell}-j_{2\ell-1}\geq t,$
 since $w\in G(B_t(u)).$ If $x_{j_{2\ell}}\mid v,$ then $j_{2\ell+1}-j_{2\ell-1}\geq j_{2\ell+1}-j_{2\ell}\geq t$ since $v\in G(B_t(u))$.
\end{proof}

Let $R$ be the polynomial ring $K[t_v\; | \; v \in G(B_t(u))]$, and $\phi: R \rightarrow K[B_t(u)]$ be the $K$-algebra homomorphism which maps $t_v$ to $v$ for all $v \in G(B_t(u))$. We denote by $J_u$, the kernel of $\phi$.


By using properties of algebras generated by sortable sets of monomials (\cite{St95} or \cite[Theorem 6.16]{EHBook}), we obtain the following result.

\begin{Theorem}\label{algebra}
The set of binomials $\MG=\{t_ut_v-t_{u^\prime}t_{v^\prime}: (u,v) \text{ unsorted }, (u^\prime, v^\prime)=\sort(u,v)\}$ is a
Gr\"obner basis of the toric ideal $J_u$.
\end{Theorem}

Since an  algebra whose defining  ideal has a quadratic Gr\"obner basis is Koszul, we get the following corollary of the above theorem.

\begin{Corollary}
$K[B_t(u)]$ is Koszul.
\end{Corollary}

Theorem~\ref{algebra} has another nice consequence.

\begin{Corollary}
$K[B_t(u)]$ is a Cohen-Macaulay normal domain.
\end{Corollary}

\begin{proof}
Theorem~\ref{algebra} shows, in particular, that  $J_u$ has a squarefree initial ideal. By a theorem due to Sturmfels \cite{St95}, it follows that $K[B_t(u)]$ is a normal domain. Next, by a theorem of Hochster \cite{Ho72}, it follows that
$K[B_t(u)]$ is Cohen-Macaulay.
\end{proof}

\section{The generic initial ideals of $t$-spread strongly stable ideals}\label{genini}

The following theorem generalizes Theorem 11.2.7 in \cite{HHBook}. For a homogeneous ideal $I\subset S=K[x_1, \ldots, x_n]$, $\Gin(I)$ stands for the generic initial ideal of $I$ with respect to the reverse lexicographic order. Throughout this section we assume that $\chara(K)=0$.

\begin{Theorem}\label{Gin}
Let $I\subset S$ be a $t$--spread strongly stable ideal. Then $I=(\Gin(I))^{\sigma^t}$. In particular, $\Gin(I)= I^{\tau^t}$ and $\Gin(I^{\sigma^t})=I$.
\end{Theorem}

\begin{proof}
We may assume  $t>0$ since the equality $I=\Gin(I)$ for strongly stable ideals is known \cite[Proposition 4.2.6]{HHBook}.

The proof is very similar to the proof of \cite[Theorem 11.2.7]{HHBook}, but we present it  in all the details for the convenience of the reader.

We use induction on the largest $\max(u)$ where $u\in G(I).$ By \cite[Lemma 11.2.8]{HHBook}, we may assume that there exists $u\in G(I)$
with $\max(u)=n.$ Following the proof of \cite[Theorem 11.2.7]{HHBook}, let $I^\prime=I:(x_n)$ and $I^{\prime\prime}$ be the ideal generated
by all the monomials $u\in G(I)$ with $\max(u)<n.$ Then, both ideals $I^\prime$ and $I^{\prime\prime}$ are $t$-spread strongly stable  and $I^{\prime\prime}\subset I\subset I^\prime.$ By the inductive hypothesis, we have
\[
I^\prime=(\Gin(I^\prime))^{\sigma^t} \text { and } I^{\prime\prime} =(\Gin(I^{\prime\prime}))^{\sigma^t}
\]
which implies that
\[I^{\prime\prime} \subset (\Gin(I))^{\sigma^t}\subset I^\prime.
\]

It is enough to show that
\begin{equation}\label{equGin1}
I\subset (\Gin(I))^{\sigma^t}.
\end{equation}
Indeed, it is well known that $I$ and $\Gin(I)$ have the same Hilbert function. By Theorem~\ref{betti}, it follows that $\Gin(I)$ and
$(\Gin(I))^{\sigma^t}$ have the same Hilbert function as well, therefore, $I$ and $(\Gin(I))^{\sigma^t}$ have the same Hilbert function. This remark together with (\ref{equGin1}) show that  $I= (\Gin(I))^{\sigma^t}.$

Let us now prove (\ref{equGin1}).  Since $I^{\prime\prime}\subset (\Gin(I))^{\sigma^t},$ we are reduced to proving that all the monomials
$u\in G(I)$ with $\max(u)=n$ belong to $(\Gin(I))^{\sigma^t}.$

Let $w_1,\ldots,w_q$ be the monomials in $G((\Gin(I))^{\sigma^t})$ with $\max(w_j)=n$ for all $j$ and $\deg w_1\leq \cdots \leq \deg w_q.$ Since
$(\Gin(I))^{\sigma^t}\subset I^\prime,$ we have $x_nw_j\in I.$ As $\max(w_j)=n$ and $I$ is a squarefree monomial ideal, it follows that
$w_j\in I$ for $1\leq j\leq q. $ This implies that, for every $1\leq j\leq q$, there exists $u_j\in G(I)$ such that $u_j \mid w_j.$ Moreover, if $\max(u_j)<n,$ then $u_j\in I^{\prime\prime} \subset (\Gin(I))^{\sigma^t},$ which is impossible since $u_j\neq w_j$ and
$w_j$ is a minimal generator of $(\Gin(I))^{\sigma^t}.$ Therefore, $\max(u_j)=n$ for $1\leq j\leq q.$

Let $u\in G(I)$ with $\max(u)=n.$ By Corollary~\ref{formula}, we have
\[\beta_{n-t(\deg u-1)-1,n-t(\deg u-1)-1+\deg u }(I)=\sum_{\stackrel{v\in G(I)}{\deg v=\deg u}}
\binom{\max(v)-t(\deg u-1)-1}{n-t(\deg u-1)-1}=
\]
\[
=|\{v\in G(I): \max(v)=n, \deg v=\deg u\}|.
\]
On the other hand,
\[\beta_{n-t(\deg u-1)-1,n-t(\deg u-1)-1+\deg u }((\Gin(I))^{\sigma^t})=\]
\[=\sum_{\stackrel{w\in G((\Gin(I))^{\sigma^t})}{\deg w=\deg u}}
\binom{\max(w)-t(\deg u-1)-1}{n-t(\deg u-1)-1}=
\]
\[=|\{w\in G((\Gin(I))^{\sigma^t}): \max(w)=n, \deg w=\deg u\}|.
\]
But, for every $ i,j,$ we have
\[\beta_{i, i+j}(I)\leq \beta_{i,i+j}(\Gin(I))=\beta_{i,i+j}((\Gin(I))^{\sigma^t}).\]
Thus, we obtain:
\begin{eqnarray}\label{equGin2}
|\{w\in G((\Gin(I))^{\sigma^t}): \max(w)=n, \deg w=\deg u\}|\geq \\ \nonumber
|\{v\in G(I): \max(v)=n, \deg v=\deg u\}|.
\end{eqnarray}
In particular, this implies that $\deg u_1$ cannot be strictly smaller than $\deg w_1,$ hence $\deg u_1=\deg w_1$ and $u_1=w_1$. In addition,  if we assume that $u_1=w_1,\ldots,u_k=w_k,$ then $\deg u_{k+1}$ cannot be strictly smaller than $\deg w_{k+1},$ which further  implies that
$u_{k+1}=w_{k+1}$ as well. Thus, $w_j\in G(I)$ for $1\leq j\leq q$, which yields
\begin{equation}\label{equGin3}
\{w\in G((\Gin(I))^{\sigma^t}): \max(w)=n\}\subset \{u\in G(I): \max(u)=n\}.
\end{equation}
However, by (\ref{equGin2}), we have
\[
|\{w\in G((\Gin(I))^{\sigma^t}): \max(w)=n\}|\geq |\{u\in G(I): \max(u)=n\}|.
\]
Consequently, in relation (\ref{equGin3}) we have equality. This shows that every monomial $u\in G(I)$ with $\max(u)=n$ belongs to
$(\Gin(I))^{\sigma^t}.$ This proves (\ref{equGin1}) and completes the proof of the theorem.
\end{proof}

\end{document}